\documentclass[12pt]{article}

\usepackage[top=1in, bottom=1in, left=1in, right=1in]{geometry}
\usepackage{amsmath}
\usepackage{amssymb}

\usepackage{amsthm}
\usepackage{enumerate}

\usepackage{graphicx}          

\theoremstyle{plain}
\newtheorem{theorem}{Theorem}
\newtheorem{lemma}{Lemma}
\newtheorem{corollary}{Corollary}

\theoremstyle{remark}
\newtheorem*{remark}{Remark}

\newcommand{\rank}{\text{rank}}

\title{Statistical Results on Filtering and Epi-convergence for Learning-Based Model Predictive Control}
\author{Anil Aswani, Humberto Gonzalez, S. Shankar Sastry, and Claire Tomlin}

\begin{document}
\maketitle

\begin{abstract}
Learning-based model predictive control (LBMPC) is a technique that provides deterministic guarantees on robustness, while statistical identification tools are used to identify richer models of the system in order to improve performance.  This technical note provides proofs that elucidate the reasons for our choice of measurement model, as well as giving proofs concerning the stochastic convergence of LBMPC.  The first part of this note discusses simultaneous state estimation and statistical identification (or learning) of unmodeled dynamics, for dynamical systems that can be described by ordinary differential equations (ODE's).  The second part provides proofs concerning the epi-convergence of different statistical estimators that can be used with the learning-based model predictive control (LBMPC) technique.  In particular, we prove results on the statistical properties of a nonparametric estimator that we have designed to have the correct deterministic and stochastic properties for numerical implementation when used in conjunction with LBMPC. 
\end{abstract}

\section{Introduction}

This technical note is meant to be understood in the context of \cite{aswani2011v1}, and it consists of two distinct parts.  Sections \ref{section:limits} and \ref{section:filter} concern simultaneous state estimation and statistical identification (or learning) of unmodeled dynamics, for dynamical systems that can be described by ordinary differential equations (ODE's).  The second part is found in Section \ref{section:epi} and provides proofs concerning the epi-convergence of different statistical estimators that can be used with the learning-based model predictive control (LBMPC) technique.

For the results on estimation and learning, we assume that for state vector $x \in \mathbb{R}^p$, control input $u \in \mathbb{R}^m$, and output $y \in \mathbb{R}^q$, the system dynamics are given by the following ODE:
\begin{equation}
\label{eqn:model}
\begin{aligned}
\dot{x} &= A_c x + B_c u + g_c(x,u) \\
y &= Cx,
\end{aligned}
\end{equation}
where $A_c,B_c,C$ are matrices of appropriate size and $g_c(x,u)$ describes the unmodeled (possibly nonlinear) dynamics.  We will assume that the control inputs generated by the model predictive control (MPC) schemes are piecewise constant
\begin{equation}
\label{eqn:input}
u(t) = u_m, \qquad \forall t \in \big[mT_u,(m+1)T_u\big),
\end{equation}
where $T_u$ is the sampling period of the input.  Note that appropriately designed MPC can generate other control schemes, such as piecewise linear inputs.

%
%

\section{Limitations on Filtering and Learning}
\label{section:limits}
We begin with a negative result about the inability of filters to estimate both the state and unmodeled dynamics, for a general system in which all states are not measured.  (This result does not apply to systems with special structure, such as in \cite{aswani_quad_2011}.)  This limitation applies to situations in which unmodeled dynamics are described by a series expansion with constant terms, and so it is relevant to a wide class of systems and filtering approaches.  

Suppose the unmodeled dynamics are parameterized as $g_c(x,u) = \gamma(x,u;\theta) + K$, where $K$ is a constant, non-zero vector and $\gamma(x,u;\theta)$ is a parametrized function such that $\gamma(x,u;\theta_0) \equiv 0$ for some parameter value $\theta_0$.  We again note that this includes the situation in which $g_c$ is given by a series expansion (e.g., Taylor polynomial, Fourier series, etc.).

The intuition is that statistical identification (or learning) of the parameters $\theta,K$ and estimation of the state $x$ can be cast into the framework of observability of an augmented dynamical system.  The augmented system has $y = Cx$ and dynamics
\begin{equation}
\label{eqn:aug}
\begin{bmatrix}\dot{x}\\\dot{K}\\\dot{\theta}\end{bmatrix} = \begin{bmatrix}A_cx + B_cu + \gamma(x,u;\theta) + K\\0\\0\end{bmatrix}.
\end{equation}
When all states are not measured and there is no special structure on $K$, then this augmented system is not observable.  This means that $(x,K,\lambda)$ cannot be simultaneously estimated using measurements of the system output $y$.  This is formalized by the following theorem.

\begin{theorem}
\label{theorem:paw}
A necessary condition for the observability (and detectability) of the system given in (\ref{eqn:aug}) with $y = Cx$ is that $\rank(C) = p$.
\end{theorem}

\begin{proof}
Suppose $\theta = \theta_0$, which makes $\gamma(x,u;\theta) \equiv 0$.  Then the system is linear and time-invariant (LTI).  Using the Popov-Belevitch-Hautus (PBH) test, the system is observable if and only if $\rank(\phi) = p + p = 2p$, for all $s \in \mathbb{C} : \text{Re}(s) \geq 0$, where
\begin{equation}
\phi = \begin{bmatrix} s\mathbb{I} - A_c & -\mathbb{I} \\ 0 & s\mathbb{I} \\ C & 0 \end{bmatrix}.
\end{equation}
If $s = 0$, then the matrices $\phi$ and $s\mathbb{I}$ are both singular, and the block structure of $\phi$ implies that $\rank(\phi) = p + \rank(C)$.  The system is not observable (and not detectable) when $\rank(C) < p$, establishing necessity.
\end{proof}

\begin{remark}
This result also applies to discrete time systems, and the proof is nearly identical.
\end{remark}

In light of this negative result concerning filters, we require that $C$ be full rank.  Without loss of generality, we assume that the full state $x$ is measured.

\section{Nonparametric Filtering for Dynamical Systems}
\label{section:filter}
The design of a Kalman filter for systems with unmodeled dynamics can be complex, and so we propose a nonparametric regression approach for estimating the state.  Available approaches include local polynomial regression (LPR) or spline-smoothing; the Savitzky-Golay filter \cite{savitzkygolay1964} is technically a finite impulse response (FIR) filter implementation of LPR.  We design a new nonparametric filter, and one advantage is of this filter is that it is easily computed because it is the weighted sums of measurements.  

An important point to note is that the statistical guarantees provided by our filter are not the same as for a Kalman filter.  The Kalman filter is defined to be consistent if its state estimates are unbiased and the true error covariance is smaller (covariance matrices are positive semi-definite, and so a partial order can be defined) than the estimated error covariance.  In our method, consistency is defined with respect to the sampling period $T_s$ of state measurements.  As $T_s \rightarrow 0$, the estimates converge to the real values in probability.  This philosophical change is necessary in order to use nonparametric statistics, otherwise we would be forced to use a parametric model of the unmodeled dynamics.

We begin with a lemma about the differentiability of the state trajectory $x(t)$ when the inputs are piecewise constant.
\begin{lemma}
Suppose $g_c(x,u)$ is $Q-1$-times differentiable.  For $m\in \mathbb{Z}$, the trajectory $x(t)$ which solves the ODE in (\ref{eqn:model}) is once-differentiable everywhere, $Q$-times differentiable at $t \neq mT_u$, and not twice-differentiable at $t = mT_u$.
\end{lemma}

\begin{proof}
The first time-derivative of $x(t)$ is given by (\ref{eqn:model}), by definition.  Because the inputs are piecewise constant (\ref{eqn:input}), the input $u(t)$ is not differentiable at $t = mT_u$.  Because the first time-derivative of $x(t)$ is a function of $u(t)$, this means that $x(t)$ is not twice-differentiable at $t = mT_u$.  Recall that $u(t)$ is constant for $t \neq mT_u$.  Thus, $u(t)$ is smooth at $t \neq mT_u$.  This implies that $x(t)$ is $Q$-times differentiable at $t \neq mT_u$, because $g_c(x,u)$ is $Q-1$-times differentiable.
\end{proof}

\begin{remark}
These qualitative features mean that we cannot use LPR methods with order higher than zero (i.e., the Nadaraya-Watson estimator) without modifying the filtering scheme.  This is an important point, because the differentiability of the trajectory $x(t)$ makes it tempting to use LPR.  Yet, no theoretical convergence guarantees can currently be made in such a situation, and the behavior of these filters may be unpredictable.
\end{remark}

In light of these restrictions, we propose a modified sampling scheme.  Recall that $T_u$ is the sampling time for control inputs, and we define $T_s$ to be the sampling time for state measurements.  We require that $kT_s = T_u$ for some $k \in \mathbb{Z}_+$, and this scheme is illustrated in Fig. \ref{fig:sampling} for the case of $k=4$.  The advantage of this sampling scheme is that the trajectory $x(t)$ is piecewise smooth (infinitely differentiable) in between the samples taken at $mT_u$, because the control input $u(t)$ is piecewise constant.  This allows us to use LPR of order higher than zero (e.g., local linear regression), which can give significant improvements in estimation error over zeroth order LPR.

\begin{figure}
  \begin{center}
      \includegraphics{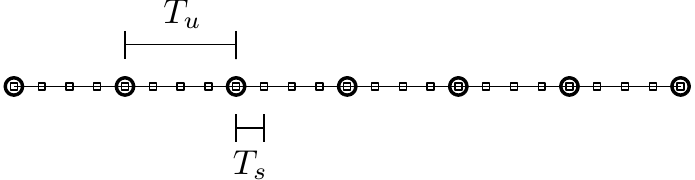}
    \caption{We use a sampling scheme with two sampling periods.  The inputs change at every $T_u$ units of time, and the states are measured every $T_s$ units of time.  In this example, $kT_s=T_u$ with $k=4$.}
    \label{fig:sampling}
  \end{center}
\end{figure}


If the trajectory of the real system is $x(t)$, then consider a measurement model
\begin{equation}
\xi_i(jT_s + mT_u) = x_i(jT_s + mT_u) + \epsilon_i, \qquad j \in \mathbb{Z}_+ : jT_s \in [0,T_u],
\end{equation}
where $\epsilon_i$ are independent and identically distributed (i.i.d.) random variables with zero mean and bounded values $l_\mu \leq [\epsilon_i]_\mu \leq s_\mu$.  The notation $[\epsilon_i]_\mu$ indicates the $\mu$-th component of the $i$-th noise vector.  Suppose that we have made measurements for $m = 0, \ldots, n$.  This measurement model corresponds to the sampling scheme seen in Figure \ref{fig:sampling}.

\subsection{Filter Design}

Suppose $\kappa(\nu)$ is a kernel function, which is a bounded even function with finite support.  We will use $\lambda,\rho$ to denote left and right differentiability, and $r$ is the polynomial order of the filter.  Let $h_{\lambda;m,i},h_{\rho;m,i} \in \mathbb{R}$ be bandwidth parameters.  Next we define a diagonal matrix $R_{m,i}$ that is used to filter to the right side of the $i$-th entry of the measurement at $t = mT_u$; its entries are given by 
\begin{equation}
R_{m,i} = \text{diag}\big\{\kappa(0), \kappa(T_s/h_{\lambda;m,i}),\ldots,\kappa(kT_s/h_{\lambda;m,i})\big\}.
\end{equation}
Similarly, we define a diagonal matrix $L_{m,i}$ that is used to filter to the left side of the $i$-th entry of the measurement at $t = mT_u$:
\begin{equation}
L_{m,i} = \text{diag}\big\{\kappa(kT_s/h_{\rho;m,i}),\ldots,\kappa(T_s/h_{\rho;m,i}),\kappa(0)\big\}.
\end{equation}
Note that the $R_{m,i}$ matrix uses the bandwidth $h_{\lambda;m,i}$, and $L_{m,i}$ uses bandwidth $h_{\rho;m,i}$.  The reason is that filtering to the right of a measurement requires left differentiability, while filtering to the left of a measurement requires right differentiability.  Lastly, we define the Vandermonde matrix
\begin{equation}
\Gamma = \begin{bmatrix} 1 & 0 & \ldots & 0 \\ 1 & T_s & \ldots & T_s^r \\ \vdots & \vdots & \vdots & \vdots \\ 1 & kT_s & \ldots & k^rT_s^r \end{bmatrix}.
\end{equation}

We are now ready to design the filter.  The filter coefficients are given by
\begin{equation}
\begin{aligned}
w_{m,i} &= e_1'(\Gamma'L_{m,i}\Gamma)^{-1}\Gamma'L_{m,i}\\
v_{m,i} &= e_1'(\Gamma'R_{m,i}\Gamma)^{-1}\Gamma'R_{m,i},
\end{aligned}
\label{eqn:wm}
\end{equation}
and $e_1$ is the unit-vector with a $1$ in the first position and zeros everywhere else.  The idea is that $w_{m,i}$ filters on the left side of $t = mT_u$ and $v_{m,i}$ filters on the right side of $t = mT_u$.  As time advances to $t = nT_u$, we first filter on the left side of $\xi(nT_u)$ (because there is no right side).  At the next point in time $t = (n+1)T_u$ we filter on both sides of $\xi(nT_u)$.  Consequently, the filter is time-varying.

Let the number within the angled brackets $\langle \cdot \rangle$ denote the (discrete) time at which the filter is computed.  The raw state estimates (for times $t = mT_u$, for $m = 0, \ldots,n$) computed at time $t = nT_u$ are given by
\begin{equation}
\begin{aligned}
&[\overline{x}_n]_i\langle n\rangle = \textstyle\sum_{j=0}^k [w_{n-1,i}]_j\xi_i(jT_s + (n-1)T_u)\\
&[\overline{x}_{n-1}]_i\langle n\rangle  = 1/2 \cdot [\overline{x}_{n-1}]_i\langle n-1\rangle + 1/2\cdot\textstyle\sum_{j=0}^k [v_{n-1,i}]_j\xi_i(jT_s + (n-1)T_u)\\
&[\overline{x}_m]_i\langle n\rangle = [\overline{x}_m]_i\langle n-1\rangle, \qquad \forall m < n-1.
\end{aligned}
\label{eqn:fil1}
\end{equation}
The state estimates are given by
\begin{equation}
\label{eqn:saturate}
[\hat{x}_m]_i\langle n\rangle = \min\Big\{\xi_i(mT_u) - l_i,\max\big\{\xi_i(mT_u)-s_i, [\overline{x}_m]_i\langle n\rangle\big\}\Big\}, \qquad \forall m.
\end{equation}
The operation in (\ref{eqn:saturate}) maintains the bounds on the noise, and it makes sure that the filter saturates if it tries to exceed the bounds of the noise.  This filtering is well-defined because of the piecewise continuity of the control input $u(t)$, and it is consistent in a pointwise sense, as the following theorem shows.

\begin{theorem}[Ruppert and Wand, 1994]
\label{theorem:filter}
If $T_u$ is fixed, $r$ is the polynomial order of the filter, and $k \rightarrow \infty$ such that $kT_s = T_u$; then, the filter defined in (\ref{eqn:fil1})-(\ref{eqn:saturate}) is consistent: $\|x_m - \hat{x}_m\langle n\rangle\| = O_p(k^{-(r+1)/(2r+3)})$.
\end{theorem}

\begin{proof}
Strictly speaking, the result in \cite{ruppertwand1994} applies to the filter defined in (\ref{eqn:wm})-(\ref{eqn:fil1}).  Consistency with respect to (\ref{eqn:saturate}) is established by noting that the bounds on the noise imply that $\|x_m - \hat{x}_m\langle n\rangle\| \leq \|x_m - \overline{x}_m\langle n\rangle\|$.
\end{proof}

\begin{remark}
Because $k = T_u/T_s$, this theorem intuitively says that the filter performs well as long as $T_s$ is much smaller than $T_u$.
\end{remark}

We also have the following lemma which discusses the finite-sample properties of (\ref{eqn:saturate}).  The intuition is that if the measurement noise is bounded and all states are measured, then the filter preserves the property that the state estimates remain within a bounded distance of the true states.  Note that the Minkowski sum \cite{schneider1993} of two sets $\mathcal{U},\mathcal{V}$ is defined as $\mathcal{U} \oplus \mathcal{V} = \{u + v : u \in \mathcal{U}; v \in \mathcal{V}\}$.
\begin{lemma}
\label{lemma:xhat}
Under the assumptions delineated above, we have that $\hat{x}_m \in x_m \oplus \mathcal{E}$, where $\mathcal{E} = \{\epsilon : l_j \leq [\epsilon]_j \leq s_j\} \oplus (-\{\epsilon : l_j \leq [\epsilon]_j \leq s_j\})$.
\end{lemma}

\begin{proof}
Note that (\ref{eqn:saturate}) enforces that $[\xi_m]_j - s_j \leq [\hat{x}_m]_j \leq [\xi_m]_j - l_j$, which can be rewritten as $\hat{x}_m \in \xi_m \oplus (-\{\epsilon : l_j \leq [\epsilon]_j \leq s_j\})$.  The bounds on the noise $[x_m]_j + l_j \leq [\xi_m]_j \leq [x_m]_j + s_j$ are equivalent to having $\xi_m \in x_m \oplus \{\epsilon : l_j \leq [\epsilon]_j \leq s_j\}$.  The result follows from properties of $\oplus$ \cite{schneider1993}.
\end{proof}

\subsection{Filter Implementation}

Because the filter is simply a weighted sum of measurements (\ref{eqn:fil1}), the largest difficulty with implementation is in computing the filter coefficients (\ref{eqn:wm}).  The first step in doing this is to choose the order of the filter.  Empirical results show that linear ($r=1$) or quadratic ($r=2$) LPR typically gives good results.  For clarity of presentation, we focus here on the case of $r=1$.  

 Having chosen the order of the filter, the next step is to compute the bandwidth parameters $h_{\lambda;m,i},h_{\rho;m,i}$.  To make the notation compact, let $?$ be a blank spot that is either replaced with $? = \rho$ or $? = \lambda$.  Using results from \cite{fangijbels1995}, it can be shown that the optimal bandwidths for $r=1$ are approximately given by
\begin{equation}
\begin{aligned}
h_{?;m,i} &= \left(\frac{a\sigma^2T_u}{2\ddot{x}_i(mT_u^?)k}\right)^{1/5}\\
a &= \int_\mathbb{R}{\kappa^2(\nu)\,d\nu}\Bigg/\left(\int_\mathbb{R}{\nu^2\kappa(\nu)\,d\nu}\right)^2,
\end{aligned}
\end{equation}
and the second time-derivative $\ddot{x}_i(mT_u^?)$ is the left-sided derivative if $?=\lambda$ (or right-sided derivative if $?=\rho$).  Unsimplified expressions for the cases $r > 1$ can be found in \cite{fangijbels1995}.  We can approximate the values of these second time-derivatives by using (\ref{eqn:model}).  More specifically, the estimated values are given by $\hat{\ddot{x}}_i(mT_u^\rho) = \left[A_c^2\xi(mT_u) + A_cB_cu_m\right]_i \label{eqn:xddr}$ and $\hat{\ddot{x}}_i(mT_u^\lambda) = \left[A_c^2\xi(mT_u) + A_cB_cu_{m-1}\right]_i$.

Because it is time consuming to compute the filter coefficients (\ref{eqn:wm}), we suggest an implementation in which they are precomputed.  Define a set $\mathcal{H} = \{h_1, \ldots, h_{max}\}$, and compute the filter coefficients for each value in $\mathcal{H}$.  Then, when we would like to filter, we estimate the time derivatives $\hat{\ddot{x}}_i(mT_u^\rho)$ and $\hat{\ddot{x}}_i(mT_u^\lambda)$, and use these to compute $h_{?;m,i}$.  The closest value in $\mathcal{H}$ is selected, and the corresponding set of precomputed filter coefficients are used to do the filtering as defined in (\ref{eqn:fil1})-(\ref{eqn:saturate}).

\section{Epi-convergence Proofs}
\label{section:epi}
We provide proofs of the theorems regarding convergence of the control law of LBMPC to an MPC that knows the unmodeled dynamics, for both the case where the oracle is parametric and the case where the oracle is nonparametric.  The key for these results is that the system trajectory must have a property called sufficient excitation (SE), which intuitively means that all modes of the system are perturbed so that they can be identified.  The theorem on convergence is trivial in the parametric case, because it results from combining two existing theorems that are valid under SE.  In the nonparametric case, we consider both a generic oracle and an oracle that we have designed, which we call the $L2$-regularized Nadaraya-Watson (L2NW) estimator.  The proofs for this case are more involved, since they require showing epi-convergence of the nonparametric oracles under the notion of SE.

\subsection{Parametric Oracle}

\begin{proof}[Proof of Theorem 5 in \cite{aswani2011v1}]
The proof simply requires application of existing theorems.  If $\lambda_n$ converges in probability to $\lambda_0$, then the result is true by Proposition 2.1 of \cite{vogellachout2003b}.  The required convergence in probability occurs under SE \cite{lai1979,jennrich1969,malinvaud1970}, and so the result trivially follows.
\end{proof}

\begin{remark}
The situation in which the states are measured with noise requires the use of the continuous mapping theorem \cite{vaart2000} taken in conjunction with Theorem \ref{theorem:filter}.  For the case where the parameters enter linearly, the hypothesis of the continuous mapping theorem is satisfied because the linear least squares estimate $\lambda_n$ is continuous with respect to the the measurements given SE \cite{lai1979}.  For the nonlinear case, we need to explicitly assume that there is a unique $\lambda_0$; under this assumption, $\lambda_0$ is a minimizer to the least squares problem with no noise, and so it is continuous with respect to measurements by the Berge maximum theorem \cite{berge1963}.  (The Berge maximum theorem gives upper hemicontinuity of $\lambda_0$, which results in continuity because $\lambda_0$ is single-valued due to its assumed uniqueness.)  This allows for the use of the continuous mapping theorem.
\end{remark}

\subsection{Nonparametric Oracle}

We first prove convergence of the control law of LBMPC that uses a generic nonparametric oracle, under an assumption of SE.  This result will then be used to prove a corresponding theorem for the case in which the oracle is taken to be the L2NW estimator.  In this section, we will refer to the functions $\tilde{\psi}_n,\phi,\tilde{\psi}_0$ that are defined in Theorem 4 of \cite{aswani2011v1}.

The first theorem we present pertains to convergence in probability of the composition of functions that individually converge in probability.  We need the following theorem in order to show epi-convergence of $\tilde{\psi}_n$, for the LBMPC problem that uses a nonparametric oracle that stochastically converges in the appropriate sense (which we will define later).

\begin{theorem}
\label{theorem:pcomp}
Let $\mathcal{X}_v \subset \mathbb{R}^\alpha$, $\mathcal{X}_w \subset \mathbb{R}^\beta$, and $\mathcal{R} \subset \mathbb{R}^\gamma$ be closed and compact sets, and assume that we have a sequence of functions $V_n(x) : \mathcal{X}_v \rightarrow \mathcal{X}_w$ and $W_n(x) : \mathcal{X}_w \rightarrow \mathcal{R}$ which converge in probability to $V(x),W(x)$ as $\sup_{x \in \mathcal{X}_v}\|V_n(x) - V(x)\| = O_p(r_n)$ and $\sup_{x \in \mathcal{X}_w}\|W_n(x) - W(x)\| = O_p(s_n)$.  If $W$ is Lipschitz continuous with constant $L_w$, then $\sup_{x \in \mathcal{X}_v}\|W_n(V_n(x)) - W(V(x))\| = O_p(c_n)$, where $c_n = \max\{r_n,s_n\}$.
\end{theorem}

\begin{proof}
Applying the triangle inequality gives
\begin{multline}
\mathbb{P}\Big(\textstyle\sup_{x \in \mathcal{X}_v}|W_n(V_n(x)) - W(V(x))|/c_n \geq \epsilon\Big) \leq \\
\mathbb{P}\Big(\textstyle\sup_{x \in \mathcal{X}_v}|W_n(V_n(x)) - W(V_n(x))|/c_n \geq \epsilon\Big) + \\ \mathbb{P}\Big(\textstyle\sup_{x \in \mathcal{X}_v}|W(V_n(x)) - W(V(x))|/c_n \geq \epsilon\Big). \label{eqn:bigone}
\end{multline}
The first term on the right in (\ref{eqn:bigone}) can be bounded as
\begin{equation}
\label{eqn:star1}
\mathbb{P}\Big(\textstyle\sup_{x \in \mathcal{X}_v}|W_n(V_n(x)) - W(V_k(x))|/c_n \geq \epsilon\Big) \leq \mathbb{P}\Big(\textstyle\sup_{x \in \mathcal{X}_w}|W_n(x) - W(x)|/c_n \geq \epsilon\Big),
\end{equation}
and so the limit of (\ref{eqn:star1}) by assumption is $\lim \mathbb{P}(\sup_{x \in \mathcal{X}_v}|W_n(V_n(x)) - W(V_n(x))|/c_n \geq \epsilon) = 0$.  The second term on the right in (\ref{eqn:bigone}) is bounded using the Lipschitz constant as
\begin{equation}
\mathbb{P}\Big(\textstyle\sup_{x \in \mathcal{X}_v}|W(V_n(x)) - W(V(x))|/c_n \geq \epsilon\Big) \leq \mathbb{P}\Big(\textstyle\sup_{x \in \mathcal{X}_v}L_w|V_n(x) - V(x)|/c_n \geq \epsilon\Big),
\end{equation}
and taking its limit gives by assumption that $\lim \mathbb{P}(\sup_{x \in \mathcal{X}_v}|W(V_n(x)) - W(V(x))|/c_n \geq \epsilon) = 0$.  The result follows by taking the limit of (\ref{eqn:bigone}) and observing that the limit is equal to zero.
\end{proof}

\begin{remark}
The theorem shows that convergence in probability is preserved under composition, but the one subtlety in the result and subsequent proof is the issue of domains of convergence.  We are composing two functions $W_n(V_n(x))$, and convergence occurs as long as the range of the function on the inside $V_n(\cdot)$ lies within the domain of convergence of the function on the outside $W_n(\cdot)$.
\end{remark}

\begin{proof}[Proof of Theorem 6 in \cite{aswani2011v1}]
Note that equality constraint in LBMPC
\begin{equation}
\tilde{x}_{n+i+1} = A\tilde{x}_{n+i} + B\check{u}_{n+i} + \mathcal{O}_n(\tilde{x}_{n+i},\check{u}_{n+i})
\end{equation}
recursively defines $\tilde{x}_{n+i+1}$, for $i = \{0,\ldots,N-1\}$, as functions of only $\tilde{x}_n$ and $c_{n+i}$.  For example, the equation for $\tilde{x}_{n+2}$ is given by
\begin{multline}
\label{eqn:long}
\tilde{x}_{n+2}(\tilde{x}_n, c;\mathcal{O}_n) = A^2\tilde{x}_n + AB(K\tilde{x}_n + c_n) + A\mathcal{O}_n(\tilde{x}_n, K\tilde{x}_n + c_n) \\
+ B(K(A\tilde{x}_n + B(K\tilde{x}_n + c_n) + \mathcal{O}_n(\tilde{x}_n, K\tilde{x}_n + c_n)) + c_{n+1}) \\
+ \mathcal{O}_n(A\tilde{x}_n + B(K\tilde{x}_n + c_n) + \mathcal{O}_n(\tilde{x}_n, K\tilde{x}_n + c_n), K(A\tilde{x}_n \\
+ B(K\tilde{x}_n + c_n) + \mathcal{O}_n(\tilde{x}_n, K\tilde{x}_n + c_n)) + c_{n+1}).
\end{multline}
Using our assumption along with the continuous mapping theorem, we have that
\begin{equation}
\sup_{x_n : \phi(x_n) \neq \emptyset}\|\tilde{x}_{n+i}(x_n,\mathcal{O}_n)-\tilde{x}_{n+i}(x_n,g)\| = O_p(r_n),
\end{equation}
where $r_n$ is the convergence rate by assumption.  Since $\psi_n$ is continuous, we can compose it with $\tilde{x}_{n+i}$ using Theorem \ref{theorem:pcomp}.  This gives that $\sup_{x_n : \phi(x_n) \neq \emptyset}\|\tilde{\psi}_n - \tilde{\psi}_0\| = O_p(r_n)$.

The last step requires showing that this condition is equivalent to lower semicontinuity in probability.  For notational convenience, we will define $c = [c_n'\ \ldots\ c_{n+N-1}']'$.  Because $\tilde{\psi}_0$ is continuous, given $\epsilon > 0$ and a point $x_0,c,\theta$, there exists a neighborhood $U\{x_0,c,\theta\}$ such that 
\begin{equation}
\label{eqn:contne}
|\tilde{\psi}_0(\zeta) - \tilde{\psi}_0(x_0,c,\theta)| < \epsilon/2,
\end{equation}
for all $\zeta \in U\{x_0,c,\theta\}$.  Now consider the expression
\begin{equation}
\alpha = \mathbb{P}\Big(\textstyle\inf_{\zeta \in U\{x_0,c,\theta\}} \tilde{\psi}_n(\zeta) < \tilde{\psi}_0(x_0,c,\theta) - \epsilon\Big) \leq \mathbb{P}\Big(\textstyle\sup_{\zeta \in U\{x_0,c,\theta\}} |\tilde{\psi}_n(\zeta) - \tilde{\psi}_0(x_0,c,\theta)| > \epsilon\Big).
\end{equation}
Using (\ref{eqn:contne}), we can further bound the expression above by
\begin{equation}
\alpha \leq \mathbb{P}\Big(\textstyle\sup_{\zeta \in U\{x_0,c[\cdot]\}} |\tilde{\psi}_n(\zeta) - \tilde{\psi}_0(\zeta)| > \epsilon/2\Big).
\end{equation}
Taking the limit, we have that $\lim \alpha = 0$, and so the result follows by applying Proposition 5.1 of \cite{vogellachout2003a}.
\end{proof}

\subsection{L2NW Estimator}

Showing that the L2NW estimator leads to convergence of the control law of LBMPC under an assumption of SE requires additional work.  For ease of reference, we give one expression of the L2NW estimator defined in \cite{aswani2011v1}.  Let $X_i = [x_i'\  u_i']'$, $Y_i = x_{i+1} - (Ax_i+Bu_i)$, and $\Xi_i = \|\xi - x_i\|^2/h^2$, where $X_i \in \mathbb{R}^{p+m}$ and $Y_i \in \mathbb{R}^p$ are data and $\xi = [x' \ u']'$ are free variables.  We define any function $\kappa : \mathbb{R} \rightarrow \mathbb{R}_+$ to be a kernel function if it has (a) finite support $\kappa(\nu) = 0$ for $|\nu| \geq 1$, (b) even symmetry $\kappa(\nu) = \kappa(-\nu)$, (c) positive values $\kappa(\nu) > 0$ for $|\nu| < 1$, (d) differentiability (i.e., the derivative $d\kappa$ exists), and (e) nonincreasing values of $\kappa(\nu)$ over $\nu \geq 0$.  The L2NW estimator is defined as
\begin{equation}
\label{eqn:rnw}
\mathcal{O}_n(x,u; X_i, Y_i) = \frac{\sum_i Y_i \kappa(\Xi_i)}{\lambda + \sum_i \kappa(\Xi_i)},
\end{equation}
where $\lambda \in \mathbb{R}_+$.  If $\lambda = 0$, then (\ref{eqn:rnw}) is simply the Nadaraya-Watson (NW) estimator.  The $\lambda$ term acts to regularize the problem and ensures differentiability.

We begin by proving a uniform version of a theorem that is called either the continuous mapping theorem \cite{vaart2000} or Slutsky's theorem \cite{bickeldoksum2007}, depending on the author.  

\begin{lemma}
\label{lemma:cmt}
Given random variables $V_n,V \in \mathcal{V}$, for all $n \in \mathbb{Z}_+$, such that $\|V_n - V\| = O_p(r_n)$;  if $L(x,v) : \mathcal{X} \times \mathcal{V} \rightarrow \mathbb{R}$ is a continuous function and $\mathcal{X},\mathcal{V}$ are compact sets, then $\sup_{x \in \mathcal{X}} |L(x, V_n) - L(x,V)| = O_p(r_n)$.
\end{lemma}

\begin{proof}
The Heine-Cantor theorem (Theorem 4.19 in \cite{rudin1964}) gives uniform continuity of $L(x,v)$ on $\mathcal{X} \times \mathcal{V}$, and this implies that for all $x$, $\|V_n - V\| > \delta > 0$ whenever $|L(x, V_n) - L(x,V)| > \epsilon > 0$.  Proceeding analogously to \cite{vaart2000}, we have
\begin{equation}
\mathbb{P}(\sup_x |L(x, V_n) - L(x,V)| > \epsilon) = \mathbb{P}(\exists x : |L(x, V_n) - L(x,V)| > \epsilon)\leq \mathbb{P}(\|V_n - V\| > \delta).
\end{equation}
The result is immediate.
\end{proof}

We can now show the first convergence result for the L2NW estimator.  Let $\hat{X}_i,\hat{Y}_i$ be defined in the same way as $X_i,Y_i$, with the change that $\hat{X}_i,\hat{Y}_i$ are defined using state estimates $\hat{x}$ instead of noiseless measurements of the state $x$.  The intuition of why this result is true is that though noise in $\hat{Y}_i$ and $\hat{X}_i$ is correlated, our filtering defined in Section \ref{section:filter} makes this correlation asymptotically insignificant.  This result can be interpreted in an instrumental variables context \cite{astrom1971,ljung1987}.

\begin{corollary}
\label{corollary:var}
If $\|\hat{X}_i - X_i\| = O_p(r_k)$, where $k \in \mathbb{N}_+$, then $\sup_{x \in \mathcal{X},u \in \mathcal{U}}\|\mathcal{O}_n(x,u;\hat{X}_i,\hat{Y}_i) - \mathcal{O}_n(x,u;X_i,Y_i)\| = O_p(r_k)$.
\end{corollary}

\begin{proof}
Define a random variable $V_k = \begin{bmatrix} \hat{X}_0'& \ldots& \hat{X}_{n-1}'&\hat{Y}_0'& \ldots&\hat{Y}_{n-1}'\end{bmatrix}'$, and let $V$ be the corresponding limiting vector.  The definition of $Y_i$ and the corollary's assumption imply that $\|\hat{Y}_i - Y_i\| = O_p(r_k)$, and so $\|V_k - V\| = O_p(n\cdot r_k)$.  

Now consider the functions $\eta = \sum_i \hat{Y}_i \kappa(\Xi_i)/n$, $\delta = \sum_i \kappa(\Xi_i)/n$, and $\rho = \eta/(\lambda/n + \delta)$.  Applying Lemma \ref{lemma:cmt} gives, $\sup_{x \in \mathcal{X},u \in \mathcal{U}}\|\eta(\xi;V_k) - \eta(\xi;V)\| = O_p(r_k)$ and $\sup_{x \in \mathcal{X},u \in \mathcal{U}}\|\delta(\xi;V_k) - \delta(\xi;V)\| = O_p(r_k)$.  Another application of Lemma \ref{lemma:cmt} gives $\sup_{x \in \mathcal{X},u \in \mathcal{U}}\|\rho(\xi;V_k) - \rho(\xi;V)\| = O_p(r_k)$.  The result follows by noting that $\mathcal{O}_n(x,u;\hat{X}_i,\hat{Y}_i) = \rho(\xi;V_k)$ and $\mathcal{O}_n(x,u;X_i,Y_i) = \rho(\xi;V)$.
\end{proof}

\begin{remark}
The variance of the NW estimator in its typical setup is known to uniformly converge at a rate no faster than $n^{-4/(p+4)}$ \cite{ruppertwand1994}.  Our result gives a nonstandard rate of convergence $r_k$, because we have a time-series setup with presmoothing to account for the errors in measurements.
\end{remark}

Convergence of an estimator is often studied by decomposing the estimation error into a bias and variance term.  For proving convergence of our L2NW estimator, we have to be careful in defining the probabilistic framework before we can decompose the error into two terms.  The $X_i$ values are not independent variables drawn from some probability distribution. They are exactly the states of a deterministic system, as it evolves in time.  In fact, if the control inputs $u_n$ are (deterministically or statistically) known for each point in time, then $X_i$ and $X_j$ are dependent for all values where $i \geq j$.  

For a nonlinear system, SE is usually defined using ergodicity or mixing, but this is hard to verify in general.  Instead, we define SE as a finite sample cover (FSC) of $\mathcal{X}$.  Let $\mathcal{B}_h(x) = \{y : \|x-y\| \leq h\}$ be a ball centered at $x$ with radius $h$, then a FSC of $\mathcal{X}$ is a set $\mathcal{S}_h = \bigcup_i \mathcal{B}_{h/2}(X_i)$ that satisfies $\mathcal{X} \subseteq \mathcal{S}_h$.  The intuition is that $\{X_i\}$ sample $\mathcal{X}$ with average, inter-sample distance less than $h/2$.  Assuming SE in the form of a FSC with asymptotically decreasing radius $h$, we can show that the control law of LBMPC that uses L2NW converges to that of an MPC that knows the true dynamics.  

Recall that $g(x,u)$ is the modeling error of the approximate linear system defined in \cite{aswani2011v1}.  We have the following result, which shows that the L2NW estimator with noiseless measurements and SE can approximate the unmodeled dynamics arbitrarily well. 

\begin{lemma}
\label{lemma:fsc}
If $g(x,u)$ is Lipschitz with constant $L$ and $\mathcal{S}_h$ is a finite sample cover of $\mathcal{Z} \subseteq \mathcal{X} \times \mathcal{U}$, then $\sup_{(x,u) \in \mathcal{Z}}\|g(x,u) - \mathcal{O}_n(x,u;X_i,Y_i)\| \leq \mu M_g + (1-\mu)L h$, where $\mu = \lambda/(\lambda+\kappa(1/2))$ and $M_g = \max \|x\| : x \in \mathcal{X}$.
\end{lemma}

\begin{proof}
Define $I = \{i : \Xi_i \leq 1\}$, and note that $\kappa(\Xi_j) = 0$ for $j \notin I$.  An alternative characteristic of the L2NW estimator is as the positively weighted average: $\mathcal{O}_n(x,u;X_i,Y_i) = w_0 \cdot 0 + \sum_{i \in I} w_i \cdot Y_i$,  where $w_0,w_i > 0$, $w_i = \kappa(\Xi_i)/(\lambda + \sum_j \kappa(\Xi_j))$, and $w_0 + \sum_i w_i = 1$.  The finite sample cover property of $\mathcal{S}_h$ implies that $\sum_j \kappa(\Xi_j) \geq \kappa(1/2)$. Noting $w_0 < \lambda/(\lambda + \kappa(1/2))$ and $Y_i = g(X_i)$, the result follows from the triangle inequality.
\end{proof}

\begin{remark}
The result shows that the L2NW estimator in our setup has bias $O(\lambda + h)$, where $\lambda = O(h)$.  This matches the bias of the NW estimator $O(h)$ in a standard setup at both interior and boundary points \cite{muller1987,ruppertwand1994}. 
\end{remark}

\begin{theorem}
\label{theorem:uniform}
Assuming $\mathcal{S}_{h_n}$ is a finite sample cover of $\mathcal{Z} \subseteq \mathcal{X} \times \mathcal{U}$, for some sequence $h_n \rightarrow 0$; $\lambda = O(h_n)$; and $k$ is a sequence such that $T_u/k \rightarrow 0$ (see Theorem \ref{theorem:filter}); then the L2NW estimator is uniformly consistent on $\mathcal{Z}$ and converges at rate
\begin{equation}
\label{eqn:cons}
\sup_{(x,u)\in\mathcal{Z}} \|g(x,u) - \mathcal{O}_n(x,u;\hat{X}_i,\hat{Y}_i)\| = O(\lambda + h_n) + O_p(k^{-(r+1)/(2r+3)}).
\end{equation}
\end{theorem}


\begin{proof}
Using the triangle inequality, the left-hand side of (\ref{eqn:cons}) is bounded by 
\begin{equation}
\sup_{(x,u)\in\mathcal{Z}} \|\mathcal{O}_n(x,u;\hat{X}_i,\hat{Y}_i) - \mathcal{O}_n(x,u;X_i,Y_i)\| + \sup_{(x,u)\in\mathcal{Z}} \|g(x,u) - \mathcal{O}_n(x,u;X_i,Y_i)\|
\end{equation}
This first term is controlled by Corollary \ref{corollary:var} and Theorem \ref{theorem:filter}, and the second is governed by Lemma \ref{lemma:fsc}.
\end{proof}

\begin{remark}
Ideally, we would like $\mathcal{Z} = \mathcal{X} \times \mathcal{U}$, but this does not always happen.  It requires that the trajectory of the system sufficiently explores the space in a manner formalized by the definition of finite sample cover.  A set $\mathcal{Z}$ which meets the assumptions of Theorem \ref{theorem:uniform} always exists, and this can be shown by construction: Given any $n > 0$, let $\mathcal{Z} = \cup_{i = 0}^{n-1} X_i$.  A better set $\mathcal{Z}$ is defined as the limit of a convergent subsequence of $X_i$, and its advantage is that the $X_i$ visit a neighborhood of the limit infinitely often.  Such a limit is guaranteed to exist by the Bolzano-Weierstrass theorem.  These two constructions mean that there is always some set on which the nonlinear dynamics $g(x,u)$ can be learned, and this set corresponds to points which the trajectory visits. 
\end{remark}

With the previous theorem, we can now show the control law of LBMPC with L2NW as oracle converges to that of MPC that knows the unmodeled dynamics, when there is SE as defined by the appropriate FSC.

\begin{proof}[Proof of Theorem 7 in \cite{aswani2011v1}]
The result directly follows from Theorem 6 in \cite{aswani2011v1}, since Theorem \ref{theorem:uniform} shows that the L2NW estimator satisfies the appropriate convergence conditions.
\end{proof}

\begin{remark}
Note that in the case of no measurement noise, the asymptotic term $O_p(k^{-(r+1)/(2r+3)})$ drops out of the corresponding expressions above.
\end{remark}

\section{Acknowledgments}

This material is based upon work supported by the National Science Foundation under Grant No. 0931843, the Army Research Laboratory under Cooperative Agreement Number W911NF-08-2-0004, the Air Force Office of Scientific Research under Agreement Number FA9550-06-1-0312, and PRET Grant 18796-S2.

\bibliographystyle{plain}        
\bibliography{safe_mpc}           

\end{document}